\documentclass[11pt]{article}
\usepackage{amsmath, amsthm, amscd, amsfonts, amssymb,color, mathtools, hyperref}
\usepackage{authblk}
\usepackage{fancyhdr,lipsum}

\newtheorem{theorem}{Theorem}[section]
\newtheorem{lemma}[theorem]{Lemma}

\newtheorem{corollary}[theorem]{Corollary}
\theoremstyle{definition}
\newtheorem{definition}[theorem]{Definition}
\newtheorem{example}[theorem]{Example}
\theoremstyle{remark}
\newtheorem{remark}[theorem]{Remark}
\numberwithin{equation}{section}
\newcommand{\reals}{\mathbb{R}}

\numberwithin{equation}{section}

\title{A CHARACTERIZATION OF NONNEGATIVITY RELATIVE TO PROPER CONES}
\author[$*$]{Chandrashekaran Arumugasamy}
\author[$**$]{Sachindranath Jayaraman}
\author[$**$]{Vatsalkumar N. Mer}
\affil[$*$]{Department of Mathematics,  School of Mathematics and Computer Sciences,\\
Central University of Tamil Nadu, Neelakudi, Kangalancherry,\\
Thiruvarur - 610 101, Tamilnadu, India}

\affil[$**$]{School of Mathematics, Indian Institute of Science Education and Research Thiruvananthapuram,\\ 
Maruthamala P.O., Vithura\\ Thiruvananthapuram -- 695 551, Kerala, India}

{
    \makeatletter
    \renewcommand\AB@affilsepx{: \protect\Affilfont}
    \makeatother

    \affil[ ]{Email ids}

    \makeatletter
    \renewcommand\AB@affilsepx{, \protect\Affilfont}
    \makeatother

    \affil[$*$]{chandrashekaran@cutn.ac.in}
    \affil[$**$]{sachindranathj@iisertvm.ac.in, vatsal.n15@iisertvm.ac.in}
}

\newcommand\shorttitle{Semipositivity of matrices}
\newcommand\authors{Chandrashekaran, Sachindranath and Vatsalkumar}

\date{\vspace{-5ex}}

\begin{document}

\maketitle
	
\begin{abstract}
Let $A$ be an $m \times n$ matrix with real entries. Given two proper cones 
$K_1$ and $K_2$ in $\reals^n$ and $\reals^m$, respectively, we say that $A$ is 
nonnegative if $A(K_1) \subseteq K_2$. $A$ is said to be semipositive if there 
exists a $x \in K_1^\circ$ such that $Ax \in K_2^\circ$. We prove that $A$ is 
nonnegative if and only if $A+B$ is semipositive for every semipositive matrix $B$. 
Applications of the above result are also brought out.
\end{abstract}

\noindent {\sc\textbf{Keywords}:} Semipositivity of matrices, nonnegative matrices, proper cones, 
semipositive cone\\
\noindent {\sc\textbf{MSC}:} 15B48, 90C33

\section{Introduction, Definition and Preliminaries}

We work with the field $\reals$ of real numbers throughout. The following notations 
will be used. $M_{m,n}(\reals)$ denotes the vector space of $m \times n$ matrices over 
$\reals$. When $m = n$, this space will be denoted by $M_n(\reals)$. 
Let us recall that a subset $K$ of $\reals^n$ is called a convex cone if $K + K \subseteq K$ 
and $\alpha K \subseteq K$ for all $\alpha \geq 0$. A convex cone $K$ is said to be proper 
if it is topologically closed, pointed ($K \cap -K = \{0\}$) and has nonempty interior 
(denoted by $K^\circ$). The dual, $K^{\ast}$, is defined as 
$K^{\ast} = \{y \in \reals^n : \langle y, x \rangle \geq 0 \ \forall x \in K\}$, where 
$\langle .,. \rangle$ denotes the usual Euclidean inner product on $\reals^n$. 
Well known examples of proper cones that occur frequently in the optimization literature 
are the nonnegative orthant $\reals^n_+$ in $\reals^n$, the Lorentz cone (also known as the 
ice-cream cone) $\mathcal{L}^n_{+}:= \{x = (x_1, \ldots, x_n)^t \in \reals^n : 
x_n \geq 0, x_n^2 - \displaystyle \sum_{i=1}^{n-1} x_i^2 \geq 0\}$, the set $\mathcal{S}^n_+$ 
of all symmetric positive semidefinite matrices in $\mathcal{S}^n$, the copositive ($\mathcal{COP}_n$) 
and completely positive ($\mathcal{CP}_n$) cones  in $\mathcal{S}^n$ and finally, the cone of 
squares $K$ in a finite dimensional Euclidean Jordan algebra (see \cite{gt} for details). 
Recall that $M_n(\reals)$ and $\mathcal{S}^n$ have the trace inner product 
$\langle X,Y \rangle = trace(Y^tX)$ and $\langle X,Y \rangle = trace(XY)$, respectively. 

\noindent
\textbf{Assumptions:} All cones in this paper are proper cones in appropriate finite dimensional 
real Hilbert spaces.

\begin{definition}\label{defn-1}
Let $K_1$ and $K_2$ be proper cones in $\reals^n$. $A \in M_n(\reals)$ is 
\begin{enumerate}
\item nonnegative (positive) if $A(K_1) \subseteq K_2 \ (A(K_1 \setminus \{0\}) 
\subseteq K_2^\circ$).
\item semipositive if there exists a $x \in K_1^\circ$ such that $Ax \in K_2^\circ$.
\item eventually nonnegative (positive) if there exists a positive integer $k_0$ such 
that $A^k$ is nonnegative 
(positive) for every $k \geq k_0$.
\end{enumerate}
\end{definition}

\noindent
Let $\pi(K_1,K_2)$ and $S(K_1,K_2)$ denote, respectively, the set of all matrices that 
are nonnegative and semipositive relative to proper cones $K_1$ and $K_2$. When $K_1 = K_2 = K$, 
we use the notation $\pi(K)$ and $S(K)$ for these sets and elements of these sets are called 
$K$-nonnegative and $K$-semipositive matrices, respectively. We write $x > 0$ to denote 
$x \in K^\circ$, the interior of $K$.

When $K$ is a proper cone in $\reals^n, \ \pi(K)$ is a proper cone in $M_n(\reals)$. 
A proof of this as well as an extensive study on the structure and properties of $\pi(K)$ 
can be found in \cite{tam} and the references cited therein. Observe that in 
Definition \ref{defn-1}, when $K = \reals^n_+$, we retrieve back the notion of a nonnegative 
matrix. A good source of reference on various applications of nonnegative matrices is the book 
by Berman and Plemmons \cite{bp-book}. Semipositivity occurs very naturally in dynamical systems, 
game theory and various other optimization problems, most notably the linear complementarity 
problem over various proper cones. For instance, given $A \in M_n(\reals)$, asymptotic 
stability (that is, the trajectory of the system from any starting point in $\reals^n$ converges 
to the origin as $t \rightarrow \infty$) of the continuous dynamical system $\dot{x} = Ax$ is equivalent to $\mathcal{S}^n_{+}$-semipositivity of the Lyapunov map $L_A$ on $\mathcal{S}^n$ 
induced by $A$, where $L_A(X) = AX + XA^t, \ X \in \mathcal{S}^n$. This is the famous Lyapunov 
theorem (for details, see the paper by Gowda and Tao \cite{gt}). We end this section by stating  
results that will be used later on. 

\begin{theorem}\label{fact-1}
For proper cones $K_1$ and $K_2$ in $\reals^n$ and $\reals^m$, respectively, and 
$A \in M_{m,n}(\reals), \ A \in \pi(K_1,K_2)$ if and only if $A^t \in \pi(K_2^{\ast},K_1^{\ast})$. 
\end{theorem}

\begin{lemma}\label{fact-2}
Let $K$ be a proper cone in $\reals^n$ and let $x \in K$. If $\langle x,y \rangle = 0$ for 
some $0 \neq y \in K^{\ast}$, then $x \notin K^\circ$.
\end{lemma}

\begin{theorem}(Theorem 2.8, \cite{chs})\label{alternative}
For proper cones $K_1$ and $K_2$ in $\reals^n$ and $\reals^m$, respectively, and 
an $m \times n$ matrix $A$, one and only one of the following
alternatives holds.
\begin{enumerate}
\item There exists $x \in K_1$ such that $A x \in K_2^\circ$. 
\item There exists $0 \neq y \in K_2^*$ such that $- A^t y \in K_1^*$.
\end{enumerate}
\end{theorem}

\section{Main Results}
We present the main results in this section. We divide this section into three parts: 
(1) A characterization of nonnegativity, (2) Applications to linear preserver problems and 
(3) Invariance of the semipositive cone of a matrix. 

\subsection{A characterization of nonnegativity relative to proper cones}\hspace*{\fill}\\

The following was proved recently by Dorsey et al in connection with strong linear preservers 
of semipositive matrices.

\begin{theorem}\label{orthant-case}
(Lemma 3.3, \cite{dgjjt}) A square matrix $A$ is $\reals^n_+$-nonnegative if and only if 
for every $\reals^n_+$-semipositive matrix $B$, the matrix $A + B$ is $\reals^n_+$-semipositive. 
\end{theorem}

Our main result of this note generalizes the above theorem to proper cones in $\reals^n$. We 
prove the result below. The first result is obvious and we skip the proof.

\begin{theorem}\label{thm-0}
Let $A \in M_{m,n}(\reals)$ and let $K_1, \ K_2$ be proper cones in $\reals^n$ and $\reals^m$, respectively. If $A \in \pi(K_1,K_2)$, then for any $B \in S(K_1,K_2), \ A+B \in S(K_1,K_2)$. 
\end{theorem}

The following theorem gives the converse of Theorem \ref{thm-0}.

\begin{theorem}\label{thm-1}
Let $A \in M_{m,n}(\reals)$ and let $K_1, \ K_2$ be proper cones in $\reals^n$ and $\reals^m$, respectively. If $A+B \in S(K_1,K_2)$ for every $B \in S(K_1,K_2)$ then $A \in \pi(K_1,K_2)$.
\end{theorem}

\begin{proof}
Suppose $A \notin \pi(K_1,K_2)$. There exists $0 \neq x \in K_1^\circ$ such that $Ax \notin K_2$. 
By the definition of $K^{\ast}$ and continuity of the inner product, there exists 
$0 \neq y \in (K_2^{\ast})^\circ$ such that $\langle Ax,y \rangle = \langle x, A^ty \rangle < 0$ 
or $\langle x, -A^ty \rangle > 0$. Let $v \in K_2^\circ$ and define $B = \dfrac{vz^t}{y^tv}$ 
where $z = -A^ty$. Note that $y^tv > 0$. Then $Bx = (z^tx) \dfrac{v}{y^tv} \in K_2^\circ$. 
We have thus verified that $B \in S(K_1,K_2)$. For $u \in K_1$ consider $(A+B)u$ and note that 
$y^T(A+B)u = 0$. That is $\langle y, (A+B)u \rangle = 0$, where $y \in (K_2^{\ast})^\circ$. 
Thus $(A+B)u \notin K_2^\circ$ (see Lemma \ref{fact-2} above). Hence $A+B \notin S(K_1,K_2)$. 
\end{proof}

The following corollary is immediate.

\begin{corollary}\label{cor-1}
Let $K$ be a proper self-dual cone in $\reals^n$. If for every $K$-semipositive matrix $B$, 
the matrix $A+B$ is $K$-semipositive, then $A$ is $K$-nonnegative.  
\end{corollary}

By taking $K_1 = \reals^n_+$ and $K_2 = \reals^m_+$, we retrieve back the result of 
Dorsey et al as in Theorem \ref{orthant-case}. 

The following is an isomorphic version of Theorem \ref{thm-1}.

\begin{theorem}\label{thm-2}
Let $K$ be a proper cone in $\reals^n$ for which Theorem \ref{thm-1} holds and let 
$K_1 \subset \reals^n$ be isomorphic to $K$ through an isomorphism $T$ with $T(K) = K_1$. 
Then $TAT^{-1} \in M_n(\reals)$ is $K_1$-nonnegative if and only if for every 
$K$-semipositive matrix $B \in M_n(\reals), \ A+B$ is $K$-semipositive.
\end{theorem}

\begin{proof}
The proof follows as $\pi(K) = T^{-1}\pi(K_1)T$, where $T$ is an isomorphism between $K$ 
and $K_1$.
\end{proof}

A few remarks are in order. 
\begin{remark}\label{rem-1}
\
\begin{enumerate}
\item Theorem \ref{thm-1} can be suitably modified for linear maps between finite dimensional 
real Hilbert spaces $V$ and $W$, equipped with proper cones $K_1$ and $K_2$, respectively. 
		
\item Let $Q$ be a nonsingular symmetric matrix with inertia 
$(n-1,0,1)$. Let $\lambda_n$ be the single negative eigenvalue of $Q$ with a normalized 
eigenvector $u_n$. Define $K:= K(Q,u_n) = \{x \in \mathbb{R}^n: x^tQx \leq 0, \ x^tu_n \geq 0\}$. 
Let $-K = K(Q,-u_n)$. It can be seen that $K(Q,\pm u_n)$ is a proper cone, known as an 
ellipsoidal cone. It is now easy to see that the Lorentz cone 
$\mathcal{L}^n_{+}$ is an example of an ellipsoidal cone. This can be seen by taking 
$Q = \begin{bmatrix*}[r]
      I_{n-1} & 0\\
	  0 & -1
     \end{bmatrix*}$ and $u_n = e_n$, the $n^{th}$ unit vector in $\mathbb{R}^n$. The 
following result from \cite{sw} will be used below. 
		
\begin{lemma}\label{ellipsoidal-cone}
(Lemma 2.7, \cite{sw}) A cone $K$ is ellipsoidal if and only if $K = T(\mathcal{L}^n_{+})$ for some 
invertible matrix $T$. 
\end{lemma}
		
\item It now follows from Theorem \ref{thm-2} and Lemma \ref{ellipsoidal-cone} that over any 
ellipsoidal cone $K$, the square matrix $TAT^{-1}$ is $K$-nonnegative if and only if for every 
$\mathcal{L}^n_{+}$-semipositive matrix $B, \ A+B$ is $\mathcal{L}^n_{+}$-semipositive.
\end{enumerate}
\end{remark}

\noindent
In what follows, we illustrate Theorem \ref{thm-1} for the Lorentz cone, as the 
computations are nontrivial and interesting.

\begin{proof} 
Suppose $A$ is not $\mathcal{L}^n_{+}$-nonnegative, but for every $\mathcal{L}^n_{+}$-semipositive 
matrix $B$, the matrix $A+B$ is $\mathcal{L}^n_{+}$-semipositive. Choose $0 \neq x \in 
\mathcal{L}^n_{+}$such that $A^t x \notin \mathcal{L}^n_{+}$. Let us discuss two cases.
	
\noindent
\underline{Case-1:} Suppose $y:= -A^t x \in \mathcal{L}^n_{+}$. Notice that $y \neq 0$. 
For, otherwise, $A^tx = 0 \in \mathcal{L}^n_{+}$. There exist 
$\alpha_i \in \mathbb{R}, \ 1 \leq i  \leq n-1$
and $\alpha_n > 0$ such that $\alpha_i x_n = y_i$ and $\alpha_n x_n < y_n$. 
Let $B = \begin{bmatrix*}[r]
	      0 & \cdots & 0 \\
	     \vdots & \cdots & \vdots \\
	      0 & \cdots & 0 \\
	     \alpha_1 & \cdots & \alpha_n
	   	 \end{bmatrix*}$. Note that $B$ is $\mathcal{L}^n_{+}$-semipositive, 
since $\alpha_n > 0$. Now $-(B+A)^t x =  \begin{bmatrix*}[r]
									   	  0 \\
										  \vdots \\
										  0 \\
										  - \alpha_n x_n +y_n
										 \end{bmatrix*} \in \mathcal{L}^n_{+}$, 
since $0 < - \alpha_n x_n +y_n$. It is now easy to see that $A+B$ is not $\mathcal{L}^n_{+}$-semipositive.\\
\noindent
\underline{Case-2:} Suppose $y:= -A^t x \notin \mathcal{L}^n_{+}$. Consider the matrix 
$B = \begin{bmatrix*}[r]
   	  0  \\
	  \vdots \\
	  0 \\
	  \frac{y^t}{x_n}
	 \end{bmatrix*}$. Suppose for every $x \in \mathcal{L}^n_{+}, \ y^t x \leq 0$. Then, 
we must have $- y = A^t x \in \mathcal{L}^n_{+}$. Since $A^t x \notin \mathcal{L}^n_{+},$ this 
is a contradiction to our assumption. Therefore $B$ is $\mathcal{L}^n_{+}$-semipositive. 
Now $-(B+A)^tx = -B^tx -A^tx = -B^tx + y$. Note that $B^t = \begin{bmatrix*}[r]
															 0 & 0 & \ldots & y/x_n
														    \end{bmatrix*}$. Therefore, 
$-B^tx + y = -y + y = 0 \in \mathcal{L}^n_{+}$. As before, it follows that $A+ B$ is not 
$\mathcal{L}^n_{+}$-semipositive.
\end{proof}

Let us also mention that a characterization of nonnegativity over the Lorentz cone 
$\mathcal{L}^n_{+}$ was given earlier by Loewy and Schneider (Theorem 2.2, \cite{ls}). 

\subsection{Application to linear preserver problems}\hspace*{\fill}\\

Besides independent interest in generalizing Lemma 3.3 of [5] over proper cones, one can 
use Theorem \ref{thm-1} to prove that a strong linear preserver of semipositivity with respect 
to a proper cone $K$ also preserves the set of nonnegative matrices over $K$. 
Let us mention a useful result before proceeding further. If $\mathcal{S}$ is a collection 
of matrices, a linear map $L$ on $M_n(\reals)$ is called an into preserver of $\mathcal{S}$ if 
$L(\mathcal{S}) \subset \mathcal{S}$ and a strong/onto linear preserver if 
$L(\mathcal{S}) = \mathcal{S}$. If $\mathcal{S} \subset M_n(\reals)$ contains a basis for $M_n(\reals)$, then a strong linear preserver $L$ of $\mathcal{S}$ is an into linear preserver 
that is invertible with $L^{-1}$ being an into linear preserver of $\mathcal{S}$ (see \cite{d} 
for details). The following results will be used in the theorem that follows.

\begin{lemma}\label{basis-result-1}
Given any proper cone $K$ in $\reals^n$, there is a basis for $M_n(\reals)$ from the (proper) 
cone $\pi(K)$.	
\end{lemma}

\begin{lemma}\label{basis-result-2}
Let $K$ be a proper cone in $\reals^n$ and let $S_1,S_2 \in \pi(\reals^n_+,K)$ be such that 
$S_1((\reals^n_+)^\circ) \subseteq K^\circ$ and $S_2$ invertible. If $B \in S(\reals^n_+)$, 
then $S_1BS_2^{-1} \in S(K)$. Consequently, $S(K)$ contains a basis for $M_n(\reals)$.
\end{lemma}

\begin{proof}
To prove the first statement, choose $x \in (\reals^n_+)^\circ$ such that 
$Bx \in (\reals^n_+)^\circ$. Take $y:= S_2x \in K$. We then have $S_1BS_2^{-1}y \in K^\circ$. 
This finishes the proof as $K$ is a proper cone. To prove the second statement, take a basis 
$\{B_1, \ldots, B_{n^2}\}$ for $M_n(\reals)$ from $S(\reals^n_+)$ (for example, the set of 
matrices with $2$ in one entry and remaining entries being $1$). Consider the collection 
$\{A_i: i = 1, \ldots n^2 \}$, where $A_i = S_1B_iS_2$ with $S_1,S_2 \in \pi(\reals^n_+,K)$ 
are both invertible and $S_1((\reals^n_+)^\circ) \subseteq K^\circ$. From the first statement 
of this result, we know that each $A_i \in S(K)$. It is now obvious that the $A_i$s form a 
basis for $M_n(\reals)$.
\end{proof}

\begin{theorem}\label{thm-1.1}
Let $K$ be a proper cone in $\reals^n$. Suppose $L$ is a strong linear preserver of $S(K)$. 
Then $L$ is a linear automorphism of $\pi(K)$.
\end{theorem}

\begin{proof}
Let $A \in \pi (K)$. To begin with, let us observe that $L$ is an invertible map as 
$M_n(\reals)$ contains a basis from $S(K)$ (refer Lemma \ref{basis-result-2} above). Then 
for any $B \in S(K), \ A+B \in S(K)$. Now $L(A+ B) = L(A) + L(B) \in S(K)$. Note that 
$L(B) \in S(K)$. Further, for any $C \in S(K)$ there exists a $B \in S(K)$ such that 
$L(B)= C$. Thus, for any $C \in S(K), \ L(A) + C \in S(K)$. It follows from 
Theorem \ref{thm-1} that $L(A) \in \pi(K)$. 

A similar argument works for $L^{-1}$ as well. Since $\pi(K)$ contains a basis 
for $M_n(\reals)$, the desired result follows.
\end{proof} 

\subsection{Semipositive cone of $A$ and its invariance}\hspace{\fill}\\

We discuss in this section, invariance of the semipositive cone of $A$. Wherever possible, 
examples are provided to substantiate our results. 
Given a proper cone $K$ in $\reals^n$ and a square matrix $A$, one can consider the set 
$$\mathcal{K}_A = \{x \in K : Ax \in K\}.$$ This set is a closed convex cone called the 
semipositive cone of $A$. Note that $\mathcal{K}_A = K \cap A^{-1}(K)$, where 
$A^{-1}(K) = \{x : Ax \in K\}$. It can be proved that $\mathcal{K}_A$ is a proper cone if $A$ 
is $K$-semipositive. Recently, in \cite{st}, Sivakumar and Tsatsomeros posed the question 
of when $A$ will leave the cone $\mathcal{K}_A$ invariant. In \cite{ac-sj-vm-1}, the authors 
answered this question affirmatively when $K = \reals^n_+$ and $A$ is semipositive. A complete 
answer in the rank one operator case and a partial answer in the case of an invertible matrix 
were also given. The proofs of these statements can be found in Theorems 2.30 and 2.31 of 
\cite{ac-sj-vm-1}. We discuss yet another case below.

We now have the following theorem.

\begin{theorem}\label{thm-3}
Let $A \in M_n(\reals)$ be $K$-semipositive. Assume that for some $j \geq 2$ 
both $A^j$ and $A^{j+1}$ map $\mathcal{K}_A^\circ$ onto itself. Then, $A$ leaves 
$\mathcal{K}_A$ invariant.
\end{theorem}

\begin{proof}
Suppose $A(\mathcal{K}_A) \nsubseteq \mathcal{L}_A$. Then, by 
Theorem \ref{thm-1}, there exists a $\mathcal{K}_A$-semipositive matrix $B$ such that $A+B$ 
is not $\mathcal{K}_A$-semipositive. Therefore, by Theorem \ref{alternative}, the system 
$$(A+B)x \in \mathcal{K}_A^\circ, x \in \mathcal{K}_A$$ has no solution. This means that the 
system $$-(A+B)^ty \in \mathcal{K}_A^{\ast}, 0 \neq y \in \mathcal{K}_A^{\ast}$$ has a solution. 
Therefore, for every $z \in \mathcal{K}_A^{\ast \ast} = \mathcal{K}_A, \ 
\langle -(A+B)^ty,z \rangle \geq 0$ or $\langle (A+B)^ty,z \rangle \leq 0$ for every 
$z \in \mathcal{K}_A$. Let $j$ be a natural number such that both $A^j$ and $A^{j+1}$ map 
$\mathcal{K}_A^\circ$ onto itself. Let $z = A^jx, x \in \mathcal{K}_A^\circ$. Then, 
$\langle y, (A+B)A^jx \rangle = \langle y,A^{j+1}x \rangle + \langle y,BA^j x \rangle \leq 0$. 
Notice that $BA^jx \in \mathcal{K}_A^\circ$. Then, the second inner product is 
positive, as $y \in \mathcal{K}_A^{\ast}$ and $BA^jx \in \mathcal{K}_A^\circ$. The first inner 
product is certainly positive as $A^{j+1}x \in \mathcal{K}_A^\circ$ and 
$0 \neq y \in \mathcal{K}_A^{\ast}$. This contradiction proves the result.
\end{proof}
	
\subsection{A few examples}\hspace*{\fill}\\
 
It is worth pointing out that when $K$ is a proper self-dual cone in $\reals^n$ 
and if $A$ is an invertible matrix such that $A^2(K) \subseteq K$, then 
$A(\mathcal{K}_A) \subseteq \mathcal{K}_A$(see Theorem 2.37, \cite{ac-sj-vm-1}), 
although the converse is not true. In Examples \ref{eg-1} and \ref{eg-2} below, 
we work with the Lorentz cone $\mathcal{L}^n_{+}$ and we shall denote in this case 
the semipositive cone by $\mathcal{L}_A$. Example \ref{eg-1} illustrates that there 
are matrices that are $K$-semipositive, but some power of $A$ is not $K$-semipositive 
and $A$ does not leave the semipositive cone invariant.

\begin{example}\label{eg-1}
Let $A = \begin{bmatrix*}[r]
          1 & -1\\
          1 & 1
         \end{bmatrix*}$ and $K = \mathcal{L}^2_{+}$. Then, $A$ is 
$\mathcal{L}^2_{+}$-semipositive, whereas $A^2$ is not. The cone 
$\mathcal{L}_A = \{x = (x_1,x_2)^t \in \reals^2 : x_1 \geq 0, x_2 \geq 0, x_2^2 \geq x_1^2\}$. 
Taking $u = (0,1)^t \in \mathcal{L}_A$, we see that $Au = (-1,1)^t \notin \mathcal{L}_A$. 
\end{example}

The following is an example of a matrix $A$ for which $A^j$ is $\mathcal{L}_A$-semipositive 
for some $j \geq 1$ (and hence, $A^j$ is $\mathcal{L}^3_{+}$-semipositive for all such $j \geq 1$), 
but $A$ does not leave the semipositive cone $\mathcal{L}_A$ invariant. Let us recall a 
definition before proceeding with the example.

\begin{definition}
Given a proper cone $K$ in $\reals^n$, we say that $A \in M_n(\reals)$ has 
\begin{itemize}
\item the $K$-Perron-Frobenius property if $\rho(A) > 0, \ \rho(A) \in \sigma(A)$ and there is 
an eigenvector $x \in K$ corresponding to $\rho(A)$.
\item the strong $K$-Perron-Frobenius property if, in addition to having the $K$-Perron-Frobenius 
property, $\rho(A)$ is a simple eigenvalue such that $\rho(A) > |\lambda|$ for all 
$|\lambda| \in \sigma(A), \ |\lambda| \neq \rho(A)$, as well as there is an eigenvector 
$x \in K^\circ$ corresponding to $\rho(A)$.
\end{itemize}
\end{definition}

\begin{example}\label{eg-2}
Let $A = \begin{bmatrix*}[r]
          1 & 1 & -1\\
          -1 & 1 & 1\\
          0 & 2 & 2
         \end{bmatrix*}$. It can be easily seen that this matrix is $\mathcal{L}^3_{+}$-semipositive 
and its associated semipositive cone $\mathcal{L}_A 
= \{x = (x_1,x_2,x_3)^t \in \mathcal{L}^3_{+}: 2(x_2^2 +x_3^2 - x_1^2) + 8x_2x_3 + 4x_1x_3 \geq 0\}$ 
is a proper cone in $\reals^3$. The matrix $A$ has the following properties.
\begin{itemize}
\item It is $\mathcal{L}_A$-semipositive, as $x^\circ = (0,0,1)^t \in \mathcal{L}_A^\circ$ and 
$Ax^\circ = (-1,1,2)^t \in \mathcal{L}_A^\circ$.
\item The spectral radius of $A, \ \rho(A) = 3.13040$ is a dominant eigenvalue of $A$ with a 
corresponding eigenvector $x = (-0.204, 0.565,1)^t$ in the interior of $\mathcal{L}_A$ 
and $A^t$ has as eigenvector $y = (-0.275, 0.586, 0.762)^t$ in the interior of 
$\mathcal{L}_A^{\ast}$, corresponding to the dominant eigenvalue $\rho(A) = 3.13040$. Therefore, 
$A$ has the strong Perron-Frobenius property relative to $\mathcal{L}_A$ and so is eventually $\mathcal{L}_A$-positive (see Theorem 7, \cite{kas-tsat}; also  see the references cited therein 
for details about the Perron-Frobenius property). Therefore, there exists a natural number $j_0$ 
such that for every $j \geq j_0, \ A^j$ is $\mathcal{L}_A$-positive. 
\item It is also eventually positive over $\mathcal{L}^3_{+}$. 
\end{itemize}

However, for $x^\circ = (1,-1,2)^t \in \mathcal{L}_A, \ Ax^\circ = (-2,0,2)^t$, which is not 
an element of $\mathcal{L}_A$. Thus, $A$ does not leave the cone $\mathcal{L}_A$ invariant. Note
that $A^2$ cannot be $\mathcal{L}^3_{+}$-nonnegative. In fact, taking the above vector 
$x^\circ \in \mathcal{L}^3_{+}$, we see that $A^2 (x^\circ) \notin \mathcal{L}^3_{+}$.
\end{example}

The above example illustrates that semipositivity of a matrix with respect to a proper 
cone $K$ together with strong Perron-Frobenius property of $A$ with respect to the 
semipositive cone $\mathcal{K}_A$ does not imply that $A(\mathcal{K}_A) \subseteq \mathcal{K}_A$. 
Other cone invariance properties of the matrix $A$ in the above example can be found in 
Examples 5 and 10 of \cite{kas-tsat}. It now follows from Theorem \ref{thm-3} that there 
is no $j$ for which both $A^j$ and $A^{j+1}$ map $\mathcal{L}_A^\circ$ 
onto itself. A direct proof of this is not obvious. We end with a few remarks. Stern and 
Tsatsomeros proved that if a square matrix $A$ is semipositive as well as has the 
$\mathcal{Z}$-property with respect to a proper cone $K$, 
then $(A + \epsilon I)^{-1} \in \pi(K)$ for all $\epsilon \geq 0$. Moreover, the determinant 
of $A$ is positive and if $A$ is also symmetric, then $A$ is positive definite 
(Corollary 1, \cite{gt}). The following examples illustrate that in Theorem \ref{thm-1}, one cannot 
restrict the matrix $B$ to positive diagonal matrices or even semipositive $M$-matrices. 

\begin{example}\label{eg-3}
Consider $A = \begin{bmatrix*}[r]
               1 & 0\\
               1 & -1
              \end{bmatrix*}$, which is not nonnegative. However, for any $\alpha > 0$, the 
matrix $A + \alpha I$ is semipositive, as for $u = (x_1,0)^t \in \reals^2$ with $x_1 > 0$, 
$(A + \alpha I)u = ((1+\alpha)x_1,x_1)^t \in (\reals^2_+)^\circ$. The same calculation proves 
that $A + B$ is semipositive for any positive diagonal matrix $B$, 
although $A$ is not nonnegative. Now let $B$ be a semipositive $M$-matrix of the form 
$\begin{bmatrix*}[r]
 \alpha & p_1\\
 p_2 & \beta
\end{bmatrix*}$, with 
$\alpha > 0, \beta > 0, p_1, p_2 \leq 0$. Let $x = (x_1,x_2)^t \in (\reals^2_+)^\circ$ be 
such that $Bx \in (\reals^2_+)^\circ$. Then, $\alpha x_1 + p_1 x_2 > 0, p_2 x_1 + \beta x_2 > 0$. 
If $x_1 \leq x_2$, then 
take $A = \begin{bmatrix*}[r]
           1 & 0\\
           1 & -1
          \end{bmatrix*}$ and if $x_1 > x_2$, then take $A = \begin{bmatrix*}[r]
                                                              1 & 0\\
                                                              -1 & 1
                                                             \end{bmatrix*}$. In both these cases, 
$A$ is not nonnegative, but $A + B$ is semipositive. 
\end{example}

As a final remark, it is worth pointing out that several interesting results on eventual 
cone invariance have been obtained recently by Kasigwa and Tsatsomeros \cite{kas-tsat}.

\vspace{1cm}
\noindent
\textbf{Acknowledgements:} The first author acknowledges the SERB, Government of India, 
for partial support in the form of a grant (Grant No. ECR/2017/000078). The third author 
acknowledges the Council of Scientific and Industrial Research (CSIR), India, for support 
in the form of Junior and Senior Research Fellowships (Award No. 09/997(0033)/2015-EMR-I).

\bibliographystyle{amsplain}

\begin{thebibliography}{99}

\bibitem{ac-sj-vm-1}
Chandrashekaran Arumugasamy, Sachindranath Jayaraman and Vatsalkumar N. Mer, 
\emph{Semipositivity of linear maps relative to proper cones in finite dimensional real Hilbert 
spaces}, Electronic Journal of Linear Algebra, \textbf{34} (2018), 304-319.

\bibitem{br-book}
R. B. Bapat and T. E. S. Raghavan, \emph{Nonnegative Matrices and Applications}, Encyclopedia of 
Mathematics and its Applications, Vol. 64, Cambridge University Press, 1997.

\bibitem{bp-book}
A. Berman and R. J. Plemmons, \emph{Nonnegative Matrices in the Mathematical Sciences}, 
SIAM Classics in Applied Mathematics, Philadelphia, 1994.

\bibitem{chs}  
B. Cain, D. Hershkowitz and H. Schneider, \emph{Theorems of the alternative for cones
and Lyapunov regularity of matrices }, Czechoslovak Mathematical Journal, \textbf{47}(122) (1997),
487-499.

\bibitem{d}
J. Dieudonn\`{e}, \emph{Sur un\`{e} g\`{e}n\`{e}ralisation du groupe orthogonal \`{a} 
quatre variables [On a generalization that an orthogonal group has four variables]}, 
Archiv der Mathematik, \textbf{1}(4) (1948/49), 282-287. 

\bibitem{dgjjt}
J. Dorsey, T. Gannon, N. Jacobson, C. R. Johnson and M. Turnansky, \emph{Linear preservers of 
semi-positive matrices}, Linear and Multilinear Algebra, \textbf{64}(9) (2016), 1853-1862.

\bibitem{gt}
M. Seetharama Gowda and J. Tao, \emph{$Z$-transformations on proper and symmetric cones: 
$Z$-transformations}, Mathematical Progremming, Series B, \textbf{117}(1-2) (2009), 195-221.

\bibitem{kas-tsat}
M. Kasigwa and M. Tsatsomeros, \emph{Eventual cone invariance}, Electronic Journal of 
Linear Algebra, \textbf{32} (2017), 204-216.

\bibitem{ls}
R. Loewy and H. Schneider, \emph{Positive operators on the $n$-dimensional ice cream cone}, 
Journal of Mathematical Analysis and Applications, \textbf{49} (1975), 375-392.

\bibitem{st}
K. C. Sivakumar and M. J. Tsatsomeros, \emph{Semipositive matrices and their semipositive cones}, 
Positivity, \textbf{22}(1) (2018), 379-398.

\bibitem{sw}
R. J. Stern and H. Wolkowicz, \emph{Exponential nonnegativity on the ice cream cone}, SIAM Journal 
of Matrix Analysis and Applications, \textbf{12}(1) (1991), 160-165.

\bibitem{tam}
B. S. Tam, \emph{On the structure of the cone of positive operators}, Linear Algebra and its 
Applications, \textbf{167}(1) (1992), 65-85.

\end{thebibliography}

\end{document}